\documentclass[11pt]{amsart}

\usepackage{amsfonts,amsmath,latexsym,amssymb,verbatim,amsbsy,amsthm}
\usepackage{dsfont,bm}
\usepackage[top=1in, bottom=1in, left=1in, right=1in]{geometry}

\usepackage[dvipsnames]{xcolor}

\usepackage[colorlinks=true, pdfstartview=FitV, linkcolor=RoyalBlue,citecolor=ForestGreen, urlcolor=blue]{hyperref}

\theoremstyle{plain}
\newtheorem{THEOREM}{Theorem}[section]

\newtheorem{theorem}[THEOREM]{Theorem}

\newtheorem{lemma}[THEOREM]{Lemma}
\newtheorem{proposition}[THEOREM]{Proposition}

\theoremstyle{definition}

\newtheorem{definition}[THEOREM]{Definition}

\theoremstyle{remark}

\newtheorem{example}[THEOREM]{Example}

\newcommand{\thm}[1]{Theorem~\ref{#1}}
\newcommand{\lem}[1]{Lemma~\ref{#1}}

\newcommand{\prop}[1]{Proposition~\ref{#1}}

\newcommand{\sect}[1]{Section~\ref{#1}}

\def \a {\alpha}
\def \b {\beta}
\def \g {\gamma}
\def \d {\delta}
\def \e {\varepsilon}

\def \k {\kappa}
\def \l {\lambda}
\def \n {\nabla}
\def \s {\sigma}
\def \t {\tau}
\def \th {\theta}

\def \bmm {{\bf m}}

\def \br {{\bf r}}

\def \bv {{\bf v}}
\def \bw {{\bf w}}
\def \bx {{\bf x}}
\def \by {{\bf y}}
\def \bz {{\bf z}}
\def \bth {\bm{\theta}}

\def \bE {{\bf E}}
\def \bF {{\bf F}}

\def \cA {\mathcal{A}}
\def \cB {\mathcal{B}}

\def \cD {\mathcal{D}}

\def \cF {\mathcal{F}}
\def \cG {\mathcal{G}}

\def \cR {\mathcal{R}}

\def \cV {\mathcal{V}}
\def \cW {\mathcal{W}}

\newcommand{\R}{\ensuremath{\mathbb{R}}}   %%% reals
\renewcommand{\S}{\ensuremath{\mathbb{S}}} %%% sphere
\newcommand{\fF}{\ensuremath{\mathbb{F}}}

\def \sign {\mathrm{sgn}}

\renewcommand{\geq}{\geqslant}

\renewcommand{\leq}{\leqslant}

\DeclareMathOperator{\conv}{conv} %
\DeclareMathOperator{\diam}{diam} %
\DeclareMathOperator{\diag}{diag} %
\DeclareMathOperator{\id}{id} %

\def \lan {\langle}
\def \ran {\rangle}

\def \p {\partial}

\def \ss {\subset}

\def \RL  {Rademacher's Lemma}
\def \DP{Duhamel's Principle}

\renewcommand{\geq}{\geqslant}

\renewcommand{\leq}{\leqslant}

\def \dxi  {\, \mbox{d}\xi}

\def \ds  {\, \mbox{d}s}

\def \dtau  {\, \mbox{d}\tau}

\def \ddt  {\frac{\mbox{d\,\,}}{\mbox{d}t}}

\begin{document}
	
\title[Grassmannian reduction of Cucker-Smale systems]{Grassmannian reduction of Cucker-Smale systems and dynamical opinion games}	

\author{Daniel Lear}

\author{David N. Reynolds}

\author{Roman Shvydkoy}

\address{Department of Mathematics, Statistics and Computer Science, University of Illinois at Chicago, 60607}

\email{lear@uic.edu}
\email{dreyno8@uic.edu}
\email{shvydkoy@uic.edu}


\date{\today}


\thanks{\textbf{Acknowledgment.}  
	This work was  supported in part by NSF
	grant DMS-1813351.}

\begin{abstract}
In this note we study a new class of alignment models with self-propulsion and Rayleigh-type friction forces, which describes the collective behavior of agents with individual characteristic parameters. We describe the long time dynamics via a new method which allows to reduce analysis from the multidimensional system to a simpler family of two-dimensional systems parametrized by a proper Grassmannian. With this method we demonstrate exponential alignment for a large (and sharp) class of initial velocity configurations confined to a sector of opening less than $\pi$. 

In the case when characteristic parameters remain frozen, the system governs dynamics of opinions for a set of players with constant convictions. Viewed as a dynamical non-cooperative game, the system is shown to possess a unique stable Nash equilibrium, which represents a settlement of opinions most agreeable to all agents. Such an agreement is furthermore shown to be a global attractor for any set of initial opinions.
\end{abstract}

\maketitle 

\section{Introduction and preliminaries}

In the mathematical theory of emergent dynamics it is often difficult to analyze systems that incorporate several counteracting forces. Yet, it is imperative to study multi-forced dynamics for application in complex biological or social systems. The famous C. Reynolds 3Zone model, which includes a trinity of repulsion-alignment-attraction forces, lays a basis for realistic computer simulation of flocks, \cite{Rey1987}. A combination of self-propulsion, Rayleigh's friction, and mutual attraction/repulsion proposed by D'Orsogna \emph{et al.} \cite{DOrsogna} produces  milling and double milling patterns -- natural formations in biological swarms such as schools of fish. Many other examples are included in these comprehensive surveys  \cite{ABFHKPPS,CCP2017,CFTV2010,DOrsogna}. 

Inclusion of alignment mechanisms into multi-forced systems has become the subject of many recent studies. The basic framework for alignment dynamics  is provided by the classical Cucker-Smale system introduced in \cite{CS2007a,CS2007b}:
 \begin{equation} \label{e:CS}	\left\{
  \begin{split}
	\dot{\bx}_i&=\bv_i, \quad  \bx_i  \in \R^n,\\
	\dot{\bv}_i&=\sum_{j=1}^Nm_j\phi(\bx_i - \bx_j)(\bv_j-\bv_i)+\bF_i , \quad  \bv_i \in \R^n.
	\end{split}\right.
	\end{equation}
Here, $\bx_i$'s are agents, $m_i$'s are masses (which can be interpreted as influence strengths), $\bv_i$'s are velocities, and $\phi$ stands for a radial positive decreasing smooth function encoding communication protocol between agents. For example,
\[
\phi(r) = \frac{\l}{(1 + r^2)^{\b/2}}.
\]
Inclusion of singular and more complex topological kernels is empirically relevant as well, and has been studied in \cite{CCMP2017,DKRT2018,MLK2019,MMPZ-survey,ST-topo,ST1}. 

Dynamics under confinement or potential interaction forces $\bF_i = -\frac1N \sum_j \n_{\bx_i} U(\bx_i -\bx_j)$  was described in detail by Shu and Tadmor \cite{ShuT2019anti,ShuT2019}, see also \cite{ST-multi} for multi-scale flocks, where flocking and aggregation was established for global convex attraction potentials $U$. Earlier, a repulsive force with kinetic prefactor 
\[
\bF_i = -\frac1N \sqrt{\cV_2}\sum_{j=1}^N \n_{\bx_i} U(\bx_i -\bx_j), \quad \cV_2 = \frac{1}{N^2} \sum_{i,j=1}^N |\bv_i- \bv_j|^2,
\] 
was proposed by Cucker and Dong \cite{CD2010} as a way to avoid collisions yet achieve flocking behavior. We note that no $N$-independent results are known either for pure repulsive potentials or the  3Zone model. However, alignment with an $N$-dependent rate (reflecting a small crowd case) can be achieved for non-degenerate communication, see Kim and Peszek \cite{KP2018}, and for degenerate communication, see Dietert and Shvydkoy \cite{DScorr} and text \cite{S-book}. Closer to the subject of this present work, S.Y. Ha, T. Ha and J.H. Kim \cite{Ha-friction}, consider Cucker-Smale system \eqref{e:CS} with self-propulsion and Rayleigh's friction force, $\bF_i = (\a-\b |\bv_i|^2)\bv_i$, showing alignment of solutions lying in the positive coordinate sector $\bv_i \in \R^n_+$ under absolute communication $\inf\phi >0$. This result will be broadly extended in \thm{t:sectorial} below.

In this paper we consider the Cucker-Smale system with friction and self-propulsion forcing which takes into account individual characteristics of agents:
\begin{equation}\label{e:F}
	\bF_i = \sigma(\theta_i-|\bv_i|^p)\bv_i,
\end{equation} 
here $\s$ is a strength parameter, $\theta_i>0$ is the characteristic parameter of the $i$'th agent which reflects either a permanent or slowly changing property. We therefore supplement  system \eqref{e:CS} with a slow alignment evolution of these parameters:
\begin{equation}\label{e:th}
\dot{\th}_i=\k \sum_{j=1}^Nm_j\phi(\bx_i - \bx_j)(\theta_j-\theta_i), \quad \th_i \in \R^+.
\end{equation}
We assume that $0\leq \k\ll 1$. The power $p>0$ in the friction part of the force \eqref{e:F} is kept arbitrary to allow flexibility of interpretation of the parameters. For example, $p=1$ corresponds to characteristic speeds, $p=2$ corresponds to energies, etc.

Our study is split into two radically different cases: $\k>0$ and $\k=0$. In the case $\k>0$ the values of characteristic parameters are converging at a slow rate towards their conserved average 
\[
\ddt \bar{\th} = 0,\quad  \bar{\th} = \sum_{i=1}^N m_i \th_i.
\]
So, the forces themselves behave asymptotically similar $\bF_i \sim \s(\bar{\th} - |\bv_i|^p)\bv_i$. These counter the effects of alignment, however, when velocities point in opposite directions $\bv_i \sim - \bv_j$. If communication $\phi$ is weak, the alignment force may never overcome the self-propulsion. Such limitation on the dynamics of alignment for general solutions has already been recognized by Ha et al \cite{Ha-friction}, see Example~\ref{ex:Ha} below. However, for solutions confined to a conical sector $\bv_i \in \Sigma$ lying above a hyperspace, i.e. of angular opening less than $\pi$ (this condition is preserved by the sectorial maximum principle, see \lem{l:sectormax}), such configuration isn't possible. Hence, alignment may in fact overpower. Indeed, we prove exactly this statement under the classical Cucker-Smale assumption on communication kernel.
\begin{theorem}\label{t:sectorial} Let $\k>0$, and 
suppose the kernel satisfies
	\begin{equation}\label{e:fat}
		\phi(r) \geq \frac{\l}{(1+r^2)^{\b/2}}, \quad \b \leq 1.
	\end{equation}
Then every sectorial solution to system \eqref{e:CS}-\eqref{e:F}-\eqref{e:th} aligns and flocks exponentially fast to a pair $(\bar{\bv},\bar{\th})$ with $|\bar{\bv}| = \bar{\th}^{1/p}$:
\[
\begin{split}
\max_i \left\lbrace |\bv_i(t) - \bar{\bv}|+|\th_i(t) - \bar{\th}|\right\rbrace & \leq C e^{-\d t},\\
\diam\{\bx_i(t)\}_{i=1}^N & \leq \bar{\cD} <\infty,\  \mbox{ and } \ \bx_i(t) - \bx_j(t) \to \bar{\bx}_{ij},
\end{split}
\]
where $\d>0$ is proportional to $\k$. 
\end{theorem}
Example~\ref{ex:fat} shows that the fat tail condition \eqref{e:fat} is sharp for sectorial solutions, just as it is sharp in the classical case, see \cite{S-book}. Note that due to lack of conservation of momentum, the direction of the vector $\bar{\bv}$ cannot be determined from initial data, and is thus an emergent property of dynamics. 

The proof of \thm{t:sectorial} is based on a new Grassmannian reduction method. We seek to eliminate the longitudinal forcing $\bF_i$ by considering a system for angles $\cos(\g_{ij}) = \bv_i \cdot \bv_j / |\bv_i||\bv_j|$. The resulting alignment term, however, loses the same diffusion property as the original one in multi-D, but it does retain it in 2D. We thus project the system onto 2D planes $\Pi$, obtaining a family of Cucker-Smale type systems, parametrized by a subgroup of the Grassmannian 
$\cG(2,n)$. We achieve angular alignment on each plane, and conclude by a  geometric argument uplifting the dynamics to the original system. With the angular alignment at hand we make another radial reduction of the system by looking at the evolution of $\cR = \max_{i',i''} \frac{|\bv_{i''}|^2}{|\bv_{i'}|^2}$. Bootstrapping on the rate of convergence $\cR \to 1$ we achieve exponential radial alignment as well.

For permanent characteristic parameters, when $\k = 0$, the same Grassmannian reduction applies to show that all sectorial solutions converge to a fixed ray $\bv_i/ |\bv_i| \to \tilde{\bv} \in \S^{n-1}$. Thus, the dynamics becomes essentially 1-dimensional, tracking the values of speeds $|\bv_i| = y_i>0$. The rest of the dynamics becomes radically different from the $k>0$ case. Since $\th_i$'s remain constant one does not expect to obtain convergence to the same values $y_i \to y_i^*$, thus the flock will inevitably spread out. So, to preserve the balance of forces one is lead to consider strong communication at a long range, $\inf \phi >0$. For simplicity we further strip the dependence on position variables $\bx_i$ by considering uniform kernel $\phi = 1$, in which case the system reduces to
\begin{equation}\label{e:opinionYintro}
 \dot{y}_i = \sum_{k=1}^N m_k (y_k - y_i ) + \s (\th_i - y_i^p) y_i.
\end{equation}
This scalar system represents a model of dynamical game of opinions in which  opinions $y_i$'s change to achieve the best agreement given a varied set of conviction values $\th_i$'s. Such an agreement $\by^*$ can be interpreted as a Nash equilibrium, see \cite{Nash}, for the non-cooperative game of $N$ players aiming to maximize the corresponding pay-offs
\begin{equation}\label{e:payoff}
p_i(\by) = \s \left( \frac12 \th_i y_i^2 - \frac{1}{p+2} y_i^{p+2}\right) - \frac{M}{2} \left( \bar{y} - y_i\right)^2, \quad \bar{y} = \frac1M \sum_j m_j y_j.
\end{equation}
Our main result shows that $\by^*$ exists, is unique and is a global attractor for the system.
\begin{theorem}\label{t:opinion} For any positive set of parameters $(\bth,\bmm, \s)\in  \R^N_+ \times \R^N_+ \times \R_+$ there exists a unique stable Nash equilibrium $\by^* \in \R^N_+$ of system \eqref{e:opinionYintro}. Moreover, any positive solution $\by(t) \in \R^N_+$ converges to $\by^*$ as $t \to \infty$.
\end{theorem}
In \thm{t:unique} we show that in fact \eqref{e:opinionYintro} captures dynamics of the general multi-dimensional case: any sectorial solution to \eqref{e:CS} - \eqref{e:F} with $\k = 0$ and $\phi \equiv 1$ settles along a ray with speeds given by the values of the Nash equilibrium: $\bv_i \to y_i^*\, \tilde{\bv}$,\, $\tilde{\bv}\in \S^{n-1}$.

The proof incorporates several new tools in the context of collective behavior models including the Brouwer topological degree and Lojasiewicz-Simon gradient method. We note that the convergence  is not a consequence of the classical Lyapunov theory as the orbits may undergo slow transient dynamics before settling near the equilibrium.

The paper is organized as follows. In \sect{s:prelim} we explore basic properties of the system and its solutions, including the important Sectorial Maximum and Minimum Principles. In \sect{s:abs} we consider the general case of absolute communication $\inf \phi >0$ and positive coupling $\k>0$, and show unconditional alignment of any solution if the communication is strong enough to overcome adverse velocity directions, see \thm{t:altR}. Our scheme here follows that of \cite{Ha-friction} but with generalizations related to arbitrary characteristic parameters and $p>0$.  The main result of \sect{s:sect} is alignment of sectorial solutions and proof of \thm{t:sectorial} via Grassmannian reduction. \sect{s:opinion} is devoted to the study of opinion dynamics and proof of \thm{t:opinion}.

We commonly use bold notation for various vectorial quantities, such as $\bv,\bx\in \R^n$, but also for collective designation of parameters of the system
\[
\bth = (\th_1,\ldots,\th_N), \quad \bmm = (m_1, \ldots, m_N).
\]
The following shortcuts are used throughout
\begin{equation*}
\bx_{ij} = \bx_i-\bx_j, \quad  \bv_{ij} = \bv_i-\bv_j, \quad  \phi_{ij} =	\phi(\bx_i-\bx_j).
\end{equation*}
We denote by $|\cdot|$ the Euclidean norm on $\R^n$ for definiteness, although our results do not depend on this particular choice. And we use $\| \cdot\|$ to denote the Euclidean norm on $\R^N$.  We use tildes as in $\tilde{\br} = \frac{\br}{|\br|}$ to denote unit vectors.

\section{Preliminaries, sectorial principles}\label{s:prelim}
Let us note a few basic properties of the system \eqref{e:CS}-\eqref{e:F}-\eqref{e:th}, which we will write in full here for future reference
\begin{equation}\label{e:CSF}	
 \left\{
\begin{split}
	\dot{\bx}_i&=\bv_i, \quad  \bx_i = \lan x_i^1,\ldots,x_i^n \ran  \in \R^n,\\
	\dot{\bv}_i&=\sum_{j=1}^Nm_j\phi(\bx_i - \bx_j)(\bv_j-\bv_i)+ \sigma(\theta_i-|\bv_i|^p)\bv_i, \quad  \bv_i = \lan v_i^1,\ldots,v_i^n \ran \in \R^n\\
	\dot{\th}_i &=\k \sum_{j=1}^Nm_j\phi(\bx_i - \bx_j)(\theta_j-\theta_i),\quad \th_i \in \R^+.
\end{split}\right.
\end{equation}

The system has mirror symmetry under transformation for each $k=1,\ldots,n$,
\begin{equation*}\label{e:mirror}
v_i^k \to -v_i^k, \qquad \th_i \to \th_i, \qquad x_i^k \to - x_i^k, \quad \forall i=1,\ldots,N.
\end{equation*}
The system is rotationally invariant: for any orthogonal transformation $U\in O(n)$,
\begin{equation}\label{e:rot}
\bv_i \to U \bv_i, \qquad \bx_i \to U \bx_i, \qquad \th_i \to \th_i
\end{equation}
is another solution.

Looking at other dynamically changing quanities, we note that the alignment force is dissipative, but the self-propulsion mechanism can increase the energy  if $\th$'s are high and $|\bv_i|$'s are low, or friction can decrease it if $|\bv_i|$'s are high. The $\th$-equation clearly obeys the Maximum Principle, while, again for the momentum equation it may fail. However, velocities do remain in certain bounds determined by initial conditions: denoting $|\bv_+| = \max_i |\bv_i|$, and $\th_+ = \max_i \th_i$ we see that  (by \RL)
\[
\ddt |\bv_+| \leq \s(\th_+(0) - |\bv_+|^p) |\bv_+|,
\]
and hence,
\begin{equation*}\label{e:upperv}
\ddt |\bv_+|^p \leq p\s(\th_+(0) - |\bv_+|^p) |\bv_+|^p.
\end{equation*}
Solving this logistic inequality directly we obtain
\begin{equation*}\label{e:max}
|\bv_+| \leq C,
\end{equation*}
where $C$ depends only on the initial conditions and parameters of the system. 

The system remains on one side of a hyperplane if initially so. Indeed, let $\ell \in \R^n$ be a unit functional whose kernel defines a hyperplane $\Pi_\ell = \ker \ell$. Suppose that all $\ell(\bv_i(0)) \geq 0$. Let us find an $i = i(t)$ such that $\ell(\bv_i) = \min_j \ell(\bv_j)$. Then by \RL\ we have
\[
\ddt \ell(\bv_i) = \sum_k m_k \phi_{ik} (\ell(\bv_k) - \ell(\bv_i)) + \sigma \ell(\bv_i)(\theta_i-|\bv_i|^p) \geq \sigma \ell(\bv_i)(\theta_i-|\bv_i|^p).
\]
Integrating we obtain
\[
\ell(\bv_i) \geq \ell(\bv_i(0)) e^{c(t)} \geq 0.
\]
An implication of this is the following principle.
\begin{lemma}[Sectorial Maximum Principle]\label{l:sectormax}
Any solution to \eqref{e:CS}-\eqref{e:F}-\eqref{e:th} with initial velocities starting in a sector 
\[
\Sigma_\cF = \bigcap_{\ell \in \cF} \{\bv: \ell(\bv) \geq 0 \}.
\]
will remain in the same sector for all times.
\end{lemma}

An important family of solutions are solutions that lie strictly above one side of a hyperplane $\ell(\bv_i)>0$. In this case, the velocities also belong to a slightly narrower sector defined by the  span of the initial velocity vectors:
\[
\R^+ \times \conv\{\bv_i\}_{i=1}^N \ss \{ \bv: \ell(\bv) \geq \e |\bv| \}, \mbox{ for some } \e>0.
\]
By rotation invariance \eqref{e:rot} we can assume without loss of generality that the solution lies above the coordinate plane $\Pi_n = \{x_n = 0\}$. Thus, we have
\begin{equation}\label{e:sect}
	v_i^n \geq \e |\bv_i|, \quad \forall i = 1,\ldots,N.
\end{equation}

\begin{definition}
	We call solutions satisfying \eqref{e:sect}  {\em sectorial}. 
\end{definition}
\begin{lemma}[Sectorial Minimum Principle]
	All sectorial solutions stay bounded away from zero,  $\min_i |\bv_i| \geq c_0$, for all time.
\end{lemma}
\begin{proof}
	Indeed, denote $v_i^n = \min_{k} v_k^n$. Then
	\begin{equation*}\label{e:lowv}
			\ddt {v}_i^n = \sum_k m_k \phi_{ik} (v_k^n - v_i^n) + \s   v_i^n (\th_i - |\bv_i|^p) \geq \s   v_i^n (\th_- - \e^{-p} (v_i^n)^p).
	\end{equation*}
	Solving this logistic ODE shows that $v_i^n$ remains bounded away from $0$.
\end{proof}

\section{Dynamics under absolute communication}\label{s:abs}
In this section we consider the case of absolute kernel with 
\[
\phi_*  = \inf_{r \geq 0} \phi(r) >0. 
\]
We define
\[
\cA=\max_{i,j=\overline{1,N}}|\bv_{ij}|, \quad \cB=\max_{i,j=\overline{1,N}}|\theta_{ij}|.
\]
We often use the maximizing functional formulation for $\cA$:
\[
\cA =\max_{|\ell|=1, i,j=\overline{1,N}}\ell(\bv_{ij}).
\]
First, looking at $\cB$ we see that,
\begin{align*}
\ddt\cB\leq - \k \phi_*M\cB, \quad M = m_1+\ldots+m_N,
\end{align*}
which implies that $\theta_i \to \bar{\theta}$ exponentially fast.

Before we proceed let us review the following illustrative example, which essentially appeared in \nolinebreak\cite{Ha-friction}.

\begin{example}\label{ex:Ha}
	Let us assume that we have a global communication $\phi \equiv \l>0$, and consider a two-agent system on the line where $v= v_1 = -v_2 >0$. Then we have the system
	\[
	\dot{x} = v, \qquad \dot{v} = - \l v + \s v(1-v^2).
	\]
	The equation can be solved explicitly.    If the Cucker-Smale communication is weak, $\l <\s$,  then the solution is given by
	\[
	v = \frac{\sqrt{1 - \frac{\l}{\s}} }{\sqrt{ 1 + c_0^2 e^{-2t (\s-\l)}}}.
	\]
	So, as we can see even global communication is not sufficient to provide alignment in this case.  
	
	When $\l = \s$, we obtain
	\[
	v(t) = \frac{v_0}{ \sqrt{ 2\s t v_0^2 + 1}}.
	\]
	Hence, the solution aligns to $0$, and does so only algebraically fast. It clearly does not converge to the natural value $v = 1$.  At the same time we can see that the agents diverge, $x(t) \sim \sqrt{t}$. So, no flocking occurs either.  
	
	Finally, when $\l > \s$ we obtain a positive alignment result
	\[
	v = \frac{c_0 e^{(\s-\l)t}\sqrt{ \frac{\l}{\s} -1 } }{\sqrt{ 1 - c_0^2 e^{2t (\s-\l)}}}.
	\]
	So, in this case $v \to 0$ exponentially fast and flocking ensues.  
\end{example} 

This example sets the stage for what happens for general solutions of the system \eqref{e:CSF} with positive coupling $\k >0$. 

\begin{lemma}
	Let $\k >0$ and $\phi_*M>\sigma\bar{\th}$, then the system \eqref{e:CSF} aligns exponentially fast,
	\begin{equation}\label{e:align}
	\mathcal{A}\leq C_0e^{-\d t},
	\end{equation}
	where $C_0, \d>0$ depend on the initial data and parameters of the system.
\end{lemma}
\begin{proof}
	We start by a traditional computation which leads to
	\begin{align*}
	\ddt\cA\leq - \phi_*M\cA+\sigma\ell[\bv_i(\theta_i-|\bv_i|^p)-\bv_j(\theta_j-|\bv_j|^p)],
	\end{align*}
	where $\ell, i,$ and $j$ are a maximizing triple for $\cA$. Now adding and subtracting $\bar{\th}$ we get,
	\begin{align*}
	\ell[\bv_i(\bar{\theta}-|\bv_i|^p)&-\bv_j(\bar{\theta}-|\bv_j|^p)]+\ell[\bv_i(\th_i-\bar{\th})-\bv_j(\th_j-\bar{\th})],\\
	&\leq \ell[\bv_i(\bar{\theta}-|\bv_i|^p)-\bv_j(\bar{\theta}-|\bv_j|^p)]+C \cB.
	\end{align*}
	Then considering the functional
	\begin{align*}
	G(\bw)=\bw(\bar{\theta}-|\bw|^p), \ \ \text{with} \ \ D_\bw G(\bw)=\bar{\th}\mathrm{Id}-|\bw|^p\mathrm{Id}-p|\bw|^{p-2} \bw \otimes \bw.
	\end{align*}
	Thus
	\begin{align*}
	\ell[\bv_i(\bar{\theta}-|\bv_i|^p)-\bv_j(\bar{\theta}-|\bv_j|^p) ] =D_\bw G(\bw)(\bv_i-\bv_j),
	\end{align*}
	for some $\bw$ on the segment $[\bv_i,\bv_j]$. Considering $\ell=\frac{\bv_i-\bv_j}{|\bv_i-\bv_j|}$ we can dismiss the entire negative definite part of $D_\bw G$, with the remaining part being $\bar{\theta}\mathrm{Id}$. Therefore,
	\begin{align*}
	\ddt\cA \leq (\sigma\bar{\th}- \phi_*M)\cA + C \cB_0 e^{-\k \phi_*M t}.
	\end{align*}
By \DP\  we conclude the lemma.
\end{proof}
Note that exponential decay of velocity variations \eqref{e:align} always implies strong flocking
\[
\diam\{\bx_i(t)\}_{i=1}^N  \leq \bar{\cD} <\infty,\  \mbox{ and } \ \bx_i(t) - \bx_j(t) \to \bar{\bx}_{ij},
\]
which is a simple consequence of integration of $\dot{\bx}_{ij} = \bv_{ij}$. 

Under the assumptions of the previous lemma we can in fact deduce much more precise information about the long time behavior of velocity field.  Let us denote by $E = E(t)$  or $\bE$ any exponentially decaying quantity.  We have so far
\[
\dot{\bv}_i = \s \bv_i (\bar{\th} - |\bv_i|^p) + \bE. 
\]
Multiplying by $p \bv_i |\bv_i|^{p-2}$ and denoting $y = |\bv_i|^p$ we obtain the following ODE
\begin{equation}\label{e:logexp}
\dot{y} = p\s y(\bar{\th}-y) + E.
\end{equation}
Although the pure forceless logistic equation is easy to solve (all positive solutions converge to $\bar{\th}$ or stay $0$ if initially zero)  the analysis of the forced ODE requires elaboration.  Let us keep in mind that we have  a solution $y$ that is a priori non-negative.  
\begin{lemma} \label{l:logE} Any non-negative solution to \eqref{e:logexp} either converges to $0$ or to $\bar{\th}$. In the latter case, convergence occurs exponentially fast.
\end{lemma}
\begin{proof}
	Indeed, suppose $y$ does not converge to $0$. Then there exists a $\d>0$ for which there exists a sequence of times $t_1,t_2,... \to \infty$ such that $y(t_i) > \d$. For $t$ large enough we have
	\[
	p\s \d (\bar{\th}-\d) + E(t) >0.
	\]
	Therefore, starting from some time $t^*$, $y(t)$ will never cross $\d$ again:  $y(t) > \d$, $t>t^*$. Solving
	\[
	\ddt (\bar{\th}-y) = - p\s y(\bar{\th}-y) + E,
	\]
	by \DP, we obtain
	\[
	\bar{\th}-y(t) = (\bar{\th}-y(t^*)) \exp\left\{ - p \s \int_{t^*}^t y(s)\ds \right\} + \int_{t^*}^t E(s) \exp\left\{ - p \s \int_{s}^t y(\tau)\dtau \right\} \ds.
	\]
	So, $|\bar{\th}-y(t)|$ is an exponentially decaying quantity. 
\end{proof}
Since $\cA \to 0$, we conclude from \lem{l:logE} that either all $\bv_i \to 0$ or all $|\bv_i| \to \bar{\th}^{1/p}$ exponentially fast.  In the latter case we obtain  $\dot{\bv}_i = \bE$, and hence all $\bv_i$'s  converge to a single vector on $\S^{n-1}$.  We therefore have a full description of the dynamics under absolute communication.
\begin{theorem}\label{t:altR}
	Let  $\k>0$, $\phi_* = \inf \phi>0$, and $ \phi_* M >\s\bar{\th}$.  Then $\cA \to 0$ exponentially fast, and either all $\bv_i \to 0$ or there exists a single vector $\bar{\bv}\in \R^n$, $|\bar{\bv}| = \bar{\th}^{1/p}$, to which all $\bv_i$ converge exponentially fast.
\end{theorem}

It is easy to see that for sectorial solutions under the conditions of \thm{t:altR} convergence to $0$ is eliminated.
	Indeed, suppose that all $\bv_i \to 0$.  Let $v^n_i = \min_j v_j^n$. Then
	\[
	\dot{v}_i^n = \sum_k m_k \phi_{ik} (v_k^n - v_i^n) + \s   v_i^n (\th_i - |\bv_i|^p) \geq \s   v_i^n (\th_i - |\bv_i|^p).
	\]
	Since all $|\bv_i|^p \to 0$, from some point on, we will find $(\th_i - |\bv_i|^p) >\bar{\th}/2$. Then
	\[
	\dot{v}_i^n  \geq  c_0  v_i^n ,
	\]
	which implies exponential growth, a contradiction.

In the next section we establish a much stronger result -- sectorial solutions align for any communication with quantitatively defined heavy tail.

\section{Sectorial solutions. Grassmannian reduction}\label{s:sect}
In this section we study sectorial solutions as defined by \eqref{e:sect}. The goal here is to prove \thm{t:sectorial} and introduce a new method of Grassmannian reduction.

The method actually applies to any Cucker-Smale system \eqref{e:CS} with a longitudinal forcing $\bF_i \times \bv_i =0$, and the nature of the force is not important. The initial step is to write down the system for the direction-vectors, $\tilde{\bv}_i = \bv_i/ |\bv_i|$,  which eliminates the force:
\begin{equation}\label{e:unitsys}
 \ddt \tilde{\bv}_i = \sum_{k=1}^N m_k \frac{|\bv_k|}{|\bv_i|} \phi_{ik} ( \id - \tilde{\bv}_i \otimes \tilde{\bv}_i ) \tilde{\bv}_k.
\end{equation}
We further write down the system for the angles $\cos(\g_{ij}) = \tilde{\bv}_i \cdot \tilde{\bv}_j$:
\begin{equation}\label{e:anglesys}
\begin{split}
\ddt \cos(\g_{ij}) & = \sum_{k=1}^N m_k \frac{|\bv_k|}{|\bv_i|} \phi_{ik} ( \cos(\g_{jk} )- \cos(\g_{ij}) \cos(\g_{ik}))\\ &+\sum_{k=1}^N m_k \frac{|\bv_k|}{|\bv_j|} \phi_{jk} ( \cos(\g_{ik} )- \cos(\g_{ij}) \cos(\g_{jk})).
\end{split}
\end{equation}
Let us note that this system in dimension 3 and higher does not have an explicit dissipative structure. However in 2D there is one:  for a planar arrangement of three angles in the upper half plane where 
$\g_{ij}$ is the largest, and $\g_{ij} < \pi - \d$ we have 
\[
\g_{ik} + \g_{jk} = \g_{ij} < \pi - \d.
\]
Then
\[
\cos(\g_{jk} )- \cos(\g_{ij}) \cos(\g_{ik}) = \cos(\g_{ij} -\g_{ik})- \cos(\g_{ij}) \cos(\g_{ik}) = \sin(\g_{ij}) \sin(\g_{ik}) \geq 0,
\]
and similarly,
\[
\cos(\g_{ik} )- \cos(\g_{ij}) \cos(\g_{jk}) \geq 0.
\]
Consequently, all the terms in both sums of \eqref{e:anglesys} are non-negative.  We now denote by $\cD = \cD(t)$ the diameter of the flock and note that $\phi_{ik} \geq \phi(\cD)$. 
Since, the velocities are also bounded from above and below, we finally obtain
\begin{multline*}
	\ddt \cos(\g_{ij}) \geq c \phi(\cD) \sum_{k=1}^N m_k (\cos(\g_{jk} )- \cos(\g_{ij}) \cos(\g_{ik}) +  \cos(\g_{ik} )- \cos(\g_{ij}) \cos(\g_{jk})) \\
	 = c \phi(\cD) \sum_{k=1}^N m_k ( \cos(\g_{ik})  + \cos(\g_{jk}) )( 1 -  \cos(\g_{ij}) ).
\end{multline*}
Now
\[
\cos(\g_{ik})  + \cos(\g_{jk}) = 2 \cos\left( \frac{\g_{ij}}{2}\right) \cos\left(\frac{\g_{ik} - \g_{jk}}{2} \right) \geq c_0,
\]
due to sectorial limitation on the angles. So,
\begin{equation*}
		\ddt \cos(\g_{ij}) \geq c M \phi(\cD) (1 -  \cos(\g_{ij})),
\end{equation*}
or
\begin{equation*}\label{e:cos}
\ddt (1-\cos(\g_{ij})) \leq - c M \phi(\cD) (1 -  \cos(\g_{ij})).
\end{equation*}

To recover similar inequality in arbitrary dimension, let us fix a 2D plane $\Pi$ containing the $x_n$-axis, and let us consider the projection of \eqref{e:unitsys} onto $\Pi$:
\begin{equation*}
	\ddt {\bv}_i^\Pi =\sum_{k=1}^Nm_k\phi_{ik}(\bv_k^\Pi -\bv_i^\Pi )+\sigma(\theta_i-|\bv_i|^p)\bv_i^\Pi.
\end{equation*}
Noting that the $n$'th coordinates of the projections remain the same as for the original vectors, all norms of ${\bv}_i^\Pi$ remain bounded above and below.  Let us now write the system for unit vectors 
\begin{equation}\label{e:unitsyspi}
\ddt \tilde{\bv}_i^\Pi = \sum_{k=1}^N m_k \frac{|\bv_k^\Pi|}{|\bv_i^\Pi|} \phi_{ik} ( \id - \tilde{\bv}_i^\Pi \otimes \tilde{\bv}_i^\Pi ) \tilde{\bv}_k^\Pi.
\end{equation}
Let us keep in mind that $\phi_{ik} = \phi(\bx_i - \bx_k)$ still depend on the original agents' coordinates.  From \eqref{e:unitsyspi} we deduce the same system  for the angles $\cos(\g^\Pi_{ij}) = \tilde{\bv}_i^\Pi \cdot \tilde{\bv}_j^\Pi$ as in the original variables:
\begin{equation}\label{e:anglesyspi}
\begin{split}
\ddt \cos(\g_{ij}^\Pi) & = \sum_{k=1}^N m_k \frac{|\bv_k^\Pi|}{|\bv_i^\Pi|} \phi_{ik} ( \cos(\g_{jk}^\Pi )- \cos(\g_{ij}^\Pi) \cos(\g_{ik}^\Pi))\\ &+\sum_{k=1}^N m_k \frac{|\bv_k^\Pi|}{|\bv_j^\Pi|} \phi_{jk} ( \cos(\g_{ik}^\Pi )- \cos(\g_{ij}^\Pi) \cos(\g_{jk}^\Pi)).
\end{split}
\end{equation}
Note that 
\[
\g^{2D} = \max_{\Pi \in \cG(1,n-1), i,j} \g^\Pi_{ij} \geq \pi - \d,
\]
where $\cG(1,n-1)$ is the space of 2D planes containing $x_n$-axis, which can be identified as the compact 1-Grassmannian manifold of $\R^{n-1}$. 

Taking now the minimum over $\Pi, i, j$, writing \eqref{e:anglesyspi} for a minimizing triple, and invoking  \RL\ we run the 2D computation above for the reduced system \eqref{e:anglesyspi}:
\begin{equation}\label{e:cos2D}
\ddt (1-\cos(\g^{2D})) \leq - c M \phi(\cD) (1 -  \cos(\g^{2D})).
\end{equation}

Now, let us observe an elementary inequality 
\begin{equation}\label{e:gg2D}
	\g = \max_{i,j} \g_{ij} \leq \g^{2D}.
\end{equation}
Indeed, let $\g = \g_{ij}$. Consider the 2-dimensional  plane $\Pi$ spanned by $x_n$-axis and $\bv_i - \bv_j$. Note that $\bv_i - \bv_j = \bv_i^\Pi -\bv_j^\Pi$. So,  considering the two  isosceles triangles spanned on $\bv_i,\bv_j$ and $\bv_i^\Pi,\bv_j^\Pi$ and applying the Cosine Theorem, we have
\[
2 (1 - \cos(\g)) = |\bv_i - \bv_j |^2 = 2 |\bv_i^\Pi|^2(1 - \cos(\g_{ij}^\Pi)) \leq 2 (1 - \cos(\g_{ij}^\Pi)) \leq  2 (1 - \cos(\g^{2D})).
\]
This proves \eqref{e:gg2D}. The opposite inequality holds up to a constant factor also
\[
c \g^{2D} \leq \g,
\]
where $c$ depends on the opening of the sector $\Sigma$. 

Let us note that \eqref{e:gg2D} still doesn't prove exponential shrinking of the solution sector, since $\cD$ depends on time and can potentially spread.  So, we come back to the original system, and derive one more equation for 
\[
\cR = \max_{i',i''} \frac{|\bv_{i''}|^2}{|\bv_{i'}|^2} = \frac{|\bv_{i(t)}|^2}{|\bv_{j(t)}|^2}.
\]
Note that $\cR$ is a priori bounded from above and below.  We write
\begin{multline*}
	\ddt \cR = \frac{1}{|\bv_{j}|^2}\sum_{k=1}^N m_k\phi_{ik} 
	  ( \bv_k \cdot \bv_i - |\bv_i|^2) +\frac{|\bv_i|^2}{|\bv_{j}|^4}\sum_{k=1}^N m_k\phi_{jk} 
	  (|\bv_{j}|^2 - \bv_k \cdot \bv_j) +\cR( \th_i - \th_j + |\bv_j|^p - |\bv_i|^p).
\end{multline*}
The first sum is negative, so we simply dismiss it. In the second sum we use
\[
|\bv_{j}|^2 - \bv_k \cdot \bv_j \leq |\bv_{j}|^2 - |\bv_{k}||\bv_{j}| \cos(\g) \leq |\bv_{j}|^2(1 - \cos(\g)) \lesssim (1 - \cos(\g)) .
\]
For the friction term we observe
\[
\cR( \th_i - \th_j + |\bv_j|^p - |\bv_i|^p) \lesssim \cB + (1- \cR^{p/2}) \lesssim \cB +(1-\cR).
\]
Thus,
\begin{equation}\label{e:R}
		\ddt (\cR - 1) \leq c_1(1 - \cos(\g))  + c_2\cB - c_3(\cR-1).
\end{equation}
Finally, we complement the system with the $\cB$-equation
\begin{equation}\label{e:B}
	\ddt\cB\leq - \k \phi(\cD) M\cB,
\end{equation}
We are now ready to prove our main result for sectorial solutions.

\begin{proof}[Proof of \thm{t:sectorial}]
We will use a bootstrap argument. Suppose $\b <1$ first. Since all velocities remain bounded, we have
\[
\cD(t) \lesssim t.
\]
Using this in \eqref{e:cos2D} and \eqref{e:gg2D} we obtain
\[
(1- \cos(\g)) + \cB \lesssim e^{- c \lan t\ran^{1-\b}}.
\]
Plugging this into the $\cR$-equation \eqref{e:R}, and using \DP, we have
\[
(\cR - 1) \lesssim e^{- c \lan t\ran^{1-\b}}.
\]
Noting the bound
\[
\ddt \cD \leq \cA \lesssim \sqrt{(\cR - 1) + (1 - \cos(\g))} \lesssim e^{- c \lan t\ran^{1-\b}},
\]
we now see that the diameter of the flock remains bounded $\cD \leq \bar{\cD}$. Going back to the system \eqref{e:cos2D}-\eqref{e:R}-\eqref{e:B} we conclude exponential decay for $(\cR - 1) + (1 - \cos(\g))+ \cB$ and hence for $\cA$. 

Next, with the obtained information we can write the equations for extreme norms
\[
\ddt |\bv_\pm|^p = p \s |\bv_\pm|^p(\bar{\th} - |\bv_\pm|^p) + E(t).
\]
Since $c\leq |\bv_\pm|^p \leq C$ we conclude that $\bar{\th} - |\bv_\pm|^p$ tend to zero exponentially fast. This immediately implies that 
\[
\ddt \bv_i = \bE_i, \quad \forall i,
\]
and hence each vector has a limit, which is common for all $\bv_i$, as $t\to \infty$.  This finishes the proof for the $\b<1$ case.

Let us turn to the case $\b = 1$. Here we make one more preliminary step: from $\cD \lesssim t$ we deduce that 
\[
(1- \cos(\g)) + \cB \lesssim \frac{1}{\lan t \ran^{\eta}}, \quad \eta>0.
\]
Hence, due to the asymptotic formula $e^{-c t} \ast \frac{1}{\lan t \ran^{\eta}} \sim \frac{1}{\lan t \ran^{\eta}}$,
\[
\cR - 1 \lesssim \frac{1}{\lan t \ran^{\eta}}.
\]
This in turn implies
\[
\cD \lesssim  \lan t \ran^{1-\frac{\eta}{2}}.
\]
Plugging into the kernel again is essentially the same as assuming that $\b < 1$, and the previous steps get repeated to finish the theorem.
\end{proof}

\begin{example}\label{ex:fat}
	A modification of the classical example shows that the fat tail condition $\b \leq 1$ is necessary. Indeed, let us consider a two-agent system with $\bv' = \lan v_1,v_2\ran$ and $\bv'' = \lan -v_1, v_2\ran$. We assume that the kernel is given by the exact power law $\phi(r) = \frac{1}{r^\b}$ for $r>r_0$. Our initial condition for coordinates of the agents $\bx' = \lan x_1,x_2\ran$, $\bx'' = \lan -x_1,x_2\ran$ will be such that $x_1(0) > 2r_0$. Then for the time period when $x_1(t) >r_0$ we have the system
	\[
	\left\{\begin{split}
	\dot{v}_1 & = - \frac{v_1}{x_1^\b} + \s v_1(1 - |\bv'|^p)\\
	\dot{v}_2 & = \s v_2(1 - |\bv'|^p).
	\end{split}\right.
	\]
	Here we assumed that $\th_1 = \th_2 = 1$. Now, if $|\bv'(0)| <1$, it will remain so by the maximum principle. Then we obtain the system
	\[
	\dot{v}_1 \geq - \frac{v_1}{x_1^\b}, \quad \dot{x}_1 = v_1.
	\]
	The system has a Lyapunov function
	\[
	L = v_1 + \frac{x_1^{1-\b}}{1-\b},
	\]
	which decays on trajectories. Thus, since $\b >1$,
	\[
	v_1(t) \geq L(t) \geq v_1(0) + \frac{x_1^{1-\b}(0)}{1-\b}.
	\]
	So, if $r_0$ is sufficiently small, we can set $1> v_1(0) > \frac{x_1^{1-\b}(0)}{\b-1}$, and $v_2(0)$ is small too in order to satisfy $|\bv'(0)| <1$. The above computation then shows that $v_1(t) > c_0 >0$. This in part implies that $x_1(t)$ is increasing and hence the condition $x_1 >r_0$ will hold indefinitely. Hence, $v_1(t) > c_0$ holds indefinitely too. This establishes misalignment.  
\end{example}

\section{Friction with frozen parameters: opinion dynamics}\label{s:opinion}

We now focus on the case when $\kappa = 0$, i.e. the parameters remain constant for each agent. It is clear that flocking itself is not possible in this case. As the agents settle to certain speeds, even if unidirectional, the speeds are not expected to be equal. So, the spread of the flock will continue to grow indefinitely. To disregard the spread of the flock, and focus solely on the interesting part of dynamics, we will assume in this section that the communication kernel is uniform $\phi \equiv 1$, so  the system becomes a first order system of ODEs:
\begin{equation}\label{e:opinionV}
	\dot{\bv}_i =\sum_{j=1}^N m_j (\bv_j-\bv_i)+\sigma(\theta_i-|\bv_i|^p)\bv_i.
	\end{equation}
In particular, for the case of one-dimensional sectorial solutions, i.e. $\bv_i = y_i >0$, the system \eqref{e:opinionV} becomes
\begin{equation}\label{e:opinionY}
 \ddt y_i = \sum_{k=1}^N m_k (y_k - y_i ) + \s (\th_i - y_i^p) y_i.
\end{equation}
In fact, it is easy to see that \eqref{e:opinionY} constitutes the core dynamics of any sectorial solution to the general system \eqref{e:opinionV}. Indeed, the method of Grassmannian reduction implies exponential shrinking of the angles:
\[
1 - \cos \g \leq c_0 e^{-M t}.
\]
Since by the sectorial maximum principle the solution sector at time $t$ is a nested family: $\Sigma(t) \subset \Sigma(s)$, $t>s$, this necessarily implies that there exists a direction 
$\tilde{\bv} \in \S^{n-1}$,
 such that 
\[
| \tilde{\bv}_i(t) -\tilde{\bv} | \leq C e^{-\d t}.
\]
Thus the original $nN$-dimensional system of  \eqref{e:opinionV} reduces to an $N$-dimensional system on the speeds $y_i = |\bv_i|$. Indeed, multiplying \eqref{e:opinionV} with $\bv_i$ and dividing by $y_i$ we obtain a perturbation of \eqref{e:opinionY}:
 \begin{equation}\label{e:fullY}
 \ddt y_i = \sum_{k=1}^N m_k (y_k - y_i ) + \s (\th_i - y_i^p) y_i + E_i(t),
\end{equation}
where $E_i$ is a generic exponentially decaying function. It is therefore expected that dynamics of solutions to \eqref{e:opinionV} are determined by the dynamics of the one dimensional system \eqref{e:opinionY}.

The study of system \eqref{e:opinionY} is interesting in its own right. It can be viewed as a model of opinion games, where $y_i$'s represent opinions and $\th_i$'s represent convictions. Convictions do not change in time, while opinions are pushed towards consensus via the alignment forces.  Long time dynamics, therefore, are expected to lead to an ``agreement'', i.e. a steady state of the system 
\begin{equation}\label{e:yi}
		 \sum_{k=1}^N m_k (y_k - y_i) + \s (\th_i - y_i^{p})y_i = 0.
	\end{equation}
A solution to \eqref{e:yi} can be interpreted as the Nash equilibrium of the opinion game, i.e. a settlement for which no agent profits from shifting into another opinion (mixed strategy in terms of \cite{Nash}). This can be expressed in terms of payoff functions \eqref{e:payoff} as a solution to 
\[
p_i(\by^*) = \max_{r^p\in [\min \th_k,\max\th_k]} p_i(y_1^*,\ldots,y_{i-1}^*, r, y_{i+1}^*,\ldots,y_N^*), \qquad \forall i =1,\ldots,N.
\]
The full justification of this interpretation will be evident from the a priori bounds on a possible solution given by \eqref{e:minmax}, and the fact that $\frac{\p^2}{\p y_i^2} p_i(\by^*) <0$ for any such solution as a consequence of \eqref{e:ylower}.

Existence of Nash equilibrium is trivial in the case of consensus of convictions  $\th_1=\ldots = \th_N$. Then $y_i = \th_i$. Generally, however, it is not evident why an equilibrium  exists as our payoff functions do not satisfy the classical convexity assumptions. Instead we resort to the use of the Brouwer topological degree, see \cite{Cronin} for background. Our main result is the following theorem.

\begin{theorem}\label{t:unique} For any positive set of parameters $(\bth,\bmm, \s)$ there exists a unique positive solution
$\by^*$ to \eqref{e:yi}, which is a locally exponentially stable equilibrium of system \eqref{e:opinionY}. The map $\by^* = \by^*(\bth,\bmm, \s): \R^N_+ \times \R^N_+ \times \R_+ \to \R^N_+$ is infinitely smooth. Moreover, any sectorial solution $\bv(t)  \in \Sigma$ to \eqref{e:opinionV} converges to the one dimensional set of vectors
\[
\bv_i(t) \to y_i^*  \tilde{\bv},
\]
for some $\tilde{\bv} \in \S^{n-1}$.  

In particular, any solution $\by\in \R^N_+$ to system  \eqref{e:opinionY} converges to $\by^*$.
\end{theorem}

The proof of this theorem involves several stages split in separate subsections below. We first discuss uniqueness and stability.

\subsection{Uniqueness and stability of Nash equilibria}
	
A rough a priori estimate on the location of any equilibrium can be obtained by simply evaluating at extreme point. Denoting by $y_+$ the maximal $y_i$, and $\th_+$ the corresponding $\th_i$ with the same index we have
\[
y_+^p \leq \th_+.
\]
Similarly,
\[
y_-^p \geq \th_-.
\]
This implies that for all $i$, the solutions settles between the extreme values of $\th$:
\begin{equation}\label{e:minmax}
	\min \th_k \leq  y_i^p \leq  \max \th_k.
\end{equation}
So, if all $\th = \th_i = \th_j$, then we obtain only one solution $y_i = \th^{1/p}$. 

In what follows we will obtain a more subtle estimate on critical points. One immediate estimate can be achieved by dropping the term $\sum_k m_k y_k$ entirely from \eqref{e:yi}:
\[
(\s \th_i -M) y_i \leq \s y_i^{p+1}.
\]
So, 
\begin{equation}\label{e:yibelow}
	y_i^p \geq \th_i - \frac{M}{\s}.
\end{equation}
Since the system is permutation  invariant, we can assume without loss of generality that $\th$'s are monotonically increasing
\[
0< \th_1 \leq \ldots \leq \th_N.  
\]

\begin{lemma}\label{l:structure}
One has $\th_1 \leq y^p_1 \leq \ldots \leq y^p_N \leq \th_N$, with the following estimate holding for all $i$:
\begin{equation}\label{e:ylower}
	y_i^p \geq \th_i + \frac{m_{\geq i} - M}{\s}, \qquad m_{\geq i} = m_i+\ldots+m_N.
\end{equation}
Moreover, 
$\th_i = \th_j$, for a pair $i\neq j$, if and only if $y_i = y_j$.
\end{lemma}
\begin{proof}
Let $i>j$. Subtracting the two equations we obtain
\[
\s(\th_i - \th_j) y_i + (\s \th_j - M)(y_i - y_j) = \s(y_i^{p+1} - y_j^{p+1}).
\]
We then have
	\begin{equation}\label{e:yijthij}
		\s(\th_i - \th_j) y_i + (\s \th_j -M)(y_i - y_j) = \s(p+1) (y_i - y_j) c^p,
	\end{equation}
for some $c$ between $y_j$ and $y_i$. If $y_i < y_j$, we divide by $y_i - y_j$ and get in view of \eqref{e:yibelow}
 	\[
(\s \th_j - M)\geq  \s(p+1)  c^p.
 	\]
This necessarily implies that  $\s \th_j - M > 0$. Using \eqref{e:yibelow} we obtain
\[
(\s \th_j - M)\geq (p+1)(\s \th_j -M),
\]
which is a contradiction.  Thus, $y_i \geq y_j$. 

Now, if $y_i = y_j$, then turning back to \eqref{e:yijthij} we conclude $\th_i = \th_j$. If $\th_i = \th_j$, but $y_i \neq y_j$, we arrive at equality
	\[
(\s \th_j - M) =   \s(p+1)  c^p,
\]
and obtain a contradiction as before. 

To obtain the estimate \eqref{e:ylower} we drop the lower terms $y_1,...,y_{i-1}$ from the average, and replace all others with $y_i$. We obtain
\[
m_{\geq i} y_i + (\s \th_i - M) y_i \leq \s y_i^{p+1},
\]
which implies \eqref{e:ylower}.
\end{proof}

We are now ready to prove existence and uniqueness of equilibria.

\begin{proposition}\label{p:unique}
For any positive set of parameters $(\bth,\bmm, \s)$ there exists a unique positive solution
$\by^*$ to \eqref{e:yi}, which is a locally exponentially stable equilibrium of system \eqref{e:opinionY}. The map $\by^* = \by^*(\bth,\bmm, \s): \R^N_+ \times \R^N_+ \times \R_+ \to \R^N_+$ is infinitely smooth.
\end{proposition}
\begin{proof}
For a fixed set of positive parameters $(\bth,\bmm, \s)\in \R^N_+ \times \R^N_+ \times \R_+$ let us consider the mapping $\fF = \fF_{\bth,\bmm, \s}: \R^N_+ \to \R^N$ defined by
\[
\fF(\by) =  \left\{ M y_i - \sum_{k=1}^N m_ky_k - \s (\th_i - y_i^p) y_i\right\}_{i=1}^N.
\]
We claim that any solution to \eqref{e:yi} is not critical for this map and it has positive Jacobian. 

Indeed, let us compute
\[
D_\by \fF(\by) = \diag\{ d_i\}_{i=1}^N -  \begin{pmatrix}
    m_1 & m_2 & \ldots & m_N \\
    \vdots & \vdots &  &\vdots \\
    m_1 & m_2 & \ldots & m_N 
  \end{pmatrix},
\]
where  
\[
d_i = M + \s(p+1)y_i^p -  \s \th_i .
\]
Via routine algebra, we compute the Jacobian
\begin{equation*}\label{e:Hess}
\det D_\by \fF_{\bth}(\by) = \prod_{i=1}^N d_i -  \sum_{k=1}^N m_k \prod_{i \neq k} d_i = \prod_{i=1}^N d_i \times \left( 1 - \sum_{k=1}^N \frac{m_k}{d_k} \right).
\end{equation*}
In order to determine the sign, let us first show that all $d_i >0$. Assume that $d_i\leq 0$ for some $i$. Then
\[
(p+1) y_i^p \leq \th_i - \frac{M}{\s}.
\]
This, in turn, implies that $\th_i - \frac{M}{\s} \geq 0$. Now, using the bound \eqref{e:ylower} we obtain
\[
(p+1) \left(\th_i + \frac{m_{\geq i} - M}{\s} \right) \leq \th_i - \frac{M}{\s},
\]
which implies
\[
p \left( \th_i - \frac{M}{\s} \right) \leq - \frac{m_{\geq i}}{\s} (p+1) < 0,
\]
which is a contradiction with the above. 

Thus, the sign of the Jacobian is determined by the sign of $1 - \sum_{k=1}^N \frac{m_k}{d_k} $.  Directly from the equations \eqref{e:yi} we obtain
\[
\frac{1}{d_i} = \frac{y_i}{M \bar{y}} - p\s \frac{y_i^{p+1}}{M \bar{y}d_i}, \qquad \bar{y} = \frac{1}{M} \sum_i m_i y_i.
\]
Then
\[
 \sum_{i=1}^N \frac{m_i}{d_i} = 1 - \frac{p \s }{M\bar{y}} \sum_{i=1}^N m_i \frac{y_i^{p+1}}{d_i} < 1,
\]
which proves the desired.

Stability follows in the same fashion once we observe that the upper left minors $M_n$, $n<N$, are given by a similar expression
\[
M_n = \prod_{i=1}^n d_i \times \left( 1 - \sum_{k=1}^n \frac{m_k}{d_k} \right).
\]
We have in this case
\[
 \sum_{k=1}^n \frac{m_k}{d_k}  < \frac{1}{\bar{y}} \frac{1}{M}\sum_{k=1}^n m_k y_k < 1.
\]

We will now focus on the Brouwer topological degree of the map $\fF$ at zero (see \cite{Cronin} for the background material). To define the degree properly, we will restrict $\fF$ to a wedge region $\cW$.
Let us denote
\begin{equation}\label{e:m-prod}
\lan \by,\bz \ran = \sum_{i=1}^N m_i y_i z_i, \quad \|\by\|_p^p =  \sum_{i=1}^N m_i y^p_i.
\end{equation}
We define 
\[
\cW = \left\{\by: y_i \geq 0,\  \e \leq \|\by\|_\infty,\  \|\by\|_{p+1} \leq R \right\},
\]
where $R>0$ is large and $\e$ is small to be determined momentarily. We verify that the image of the boundary does not contain the origin, $0 \not \in \fF(\p \cW)$. Indeed, if $y_i = 0$ for some $i$, then $\fF^i = \bar{y} >0$. Let us now compute the momentum
\[
\sum_{i=1}^N m_i \fF^i (\by) = - \s \lan \bth,\by\ran + \s \|\by\|_{p+1}^{p+1}.
\]
If $\|\by\|_{p+1} = R$,  we have the bound
\[
\geq \s \|\by\|_{p+1}^{p+1} - \s \|\by\|_{p+1}\|\bth\|_{\frac{p+1}{p}} >0,
\]
provided $R$ is large enough. If $\|\by\|_\infty = \e$, then
\[
\leq -\s \th_- \|\by\|_1 + \s \e^p \|\by\|_1 < 0,
\]
provided $\e$ is small enough.

Since the value $\mathbf{0}$ of the $\fF$-map is regular, the degree can be computed explicitly by
\begin{equation*}
	\deg\{\fF, \cW, \mathbf{0}\}  = \sum_{\by \in \fF^{-1}(\mathbf{0})} \sign\left( \det D_\by \fF(\by) \right).
\end{equation*}
Since all the signs of the Jacobian are positive, the uniqueness can be obtained by showing that $	\deg\{\fF, \cW, \mathbf{0}\}  = 1$. This is certainly true for any diagonal $\hat{\bth} = (\th,\ldots,\th)$ as we have a unique positively oriented solution in this case, see \eqref{e:minmax}.  Let us fix any such $\hat{\bth}$ and consider the homotopy of maps
\[
\fF_\t = \fF_{\t \bth + (1-\t) \hat{\bth}, \bmm,\s}, \quad 0\leq \t \leq 1.
\]
We have verified above that $\mathbf{0} \not \in \fF_{\t}(\p \cW)$ for any $\t$. Thus, the Invariance under Homotopy Principle applies. As a consequence, 
\[
\deg\{\fF_{\bth,\bmm,\s}, \cW, \mathbf{0}\} = \deg\{\fF_{\hat{\bth},\bmm,\s}, \cW, \mathbf{0}\} = 1,
\]
and the proof of uniqueness is finished.

The smoothness of $\by^*$ as a function of $(\bth,\bmm, \s)$ follows directly from the non-degeneracy established above and the Implicit Function Theorem.
\end{proof}

\subsection{Gradient structure and convergence to the Nash equilibrium} As we already noted in the beginning of this section, the long time behavior of any sectorial solution to the full system \eqref{e:opinionV} reduces to the study of positive solutions to \eqref{e:fullY}. It is natural to expect that any solution to \eqref{e:fullY} converges to the unique Nash equilibrium $\by^*$. In fact, this would be a trivial application of the Lyapunov classical theory and exponential stability established in \prop{p:unique} if the initial conditions started in a small neighborhood of $\by^*$ and if $E(t)$ was already small. For general solutions, however, we resort to a hidden gradient structure of \eqref{e:fullY}.

First, let us establish boundedness. Indeed, by the energy estimates, and dropping the dissipation term coming from the alignment entirely, we obtain
\[
\ddt \|\by\|_2^2 \leq \s \th^+\|\by\|_2^2 - \s \|\by\|_{p+2}^{p+2} + E(t) + \|\by\|_2^2 \leq  \|\by\|_2^2(\s \th^++1 - \s \|\by\|_{2}^{p}) +  E(t).
\]
If at some point of time $\|\by\|_{2}^{p} > 2(\th^+ + 1/\s)$, then
\[
\ddt \|\by\|_2^2 \leq - (\s \th^++1 ) \|\by\|_2^2 + E(t) \leq -c_0 + E(t).
\]
Thus, starting from some time $T$, when $E(t) < c_0/2$, $t>T$, such an estimate would give a decaying energy. The standard maximum principle argument concludes the proof.

To get a better understanding of the long time dynamics of solutions we rescale the system and convert it to a gradient flow. Such structure, in fact, is apparent when all masses are equal, say $m_i = \frac{1}{N}$. Then 
\begin{equation}\label{e:gradY}
\ddt \by = - \n \Phi(\by) + \bE(t),
\end{equation}
for
\begin{equation*}\label{e:Phi}
\Phi(\by) = -\frac{1}{2N}( y_1 + \ldots + y_N)^2 - \frac12 \sum_i (\s \th_i - 1)y_i^2 + \frac{\s}{p+2} \sum_i y_i^{p+2}.
\end{equation*}
For general case, we introduce new variables
\[
z_i = \sqrt{m_i} y_i.
\]
The new system takes form
\begin{equation*}\label{e:z}
\ddt z_i =  \sum_j \sqrt{m_i m_j} z_j - M z_i + \frac{\s}{m_i^{p/2}}( m_i^{p/2} \th_i - z_i^p) z_i + E_i(t).
\end{equation*}
The core of the right hand side is given by $-\n \Phi$, where 
\[
\Phi(\bz) = -\frac12 \left( \sum_j \sqrt{ m_j} z_j \right)^2 + \frac{1}{2} \sum_j(\s \th_j - M) z_j^2  +  \frac{\s}{p+2} \sum_j \frac{z_j^{p+2}}{m_j^{p/2}}.
\]
The original system is now converted into a perturbed gradient flow 
\begin{equation}\label{e:fullZ}
\ddt \bz = - \n \Phi(\bz) + \bE(t).
\end{equation}
Note that the statement of \prop{p:unique} and boundedness of solutions to \eqref{e:fullZ} translates into the new system directly from the old one. Yet, the alignment structure of the new system has been destroyed. Our key observation is that for the convergence result this structure is no longer required.

We now proceed to proving that any positive solution to \eqref{e:fullZ} converges to the unique equilibrium $\bz^*$, $z^*_i = \sqrt{m_i} y^*_i$.  The proof is based on Lojasiewicz's gradient inequality, which we recall next.
\begin{theorem}[\cite{L}] Let $\Phi$ be a real analytic function in a neighborhood $U$. Then for any $\bz_0 \in U$ there are constants $c>0$ and $\delta\in(0,1]$ and $\mu \in[1/2,1)$ such that
\begin{equation}\label{Lojasiewicz}
\|\n\Phi(\bz)\|\geq c |\Phi(\bz)-\Phi(\bz_0)|^{\mu}, \qquad \forall\, \bz \in U \text{ such that } \|\bz-\bz_0\|\leq \delta.
\end{equation}
\end{theorem}
Our potential $\Phi$ is real analytic in the positive sector, so the result applies. 

So, let us consider a positive solution $\bz \in \R^N_+$ to \eqref{e:fullZ}. Since every such solution is bounded, $\bz(t)$ has an accumulation point $\bz_0$. It is easy to see that $\bz_0 \in \R^N_+$. Indeed, by the sectorial maximum principle for the original system, the corresponding $\by$-solution will remain in a smaller sector $\Sigma \ss \R^N_+$ which does not intersect the coordinate planes. The rescaled solution $\bz$ will clearly remain in the same sector. The computation in the proof of \prop{p:unique} also shows that $\by$, and hence $\bz$ cannot approach $\mathbf{0}$. Hence, $\bz_0 \in \Sigma \backslash \{ \mathbf{0} \}$.

Let us consider an increasing sequence of times $\{t_n: n\geq 1\}$ for which $\bz(t_n) \to \bz_0$. We show that $\bz(t)$ eventually enters and remains in $B_r(\bz_0):=\{\bz\in\R^N: \|\bz-\bz_0\|<r\}$. Since $r$ is arbitrarily small, it will imply that $\bz(t) \to \bz_0$. We will also show along the way that $\ddt \bz \to 0$ along some sequence of times. This will establish that $\n \Phi (\bz_0) = 0$, and hence $\bz_0 = \bz^*$.

The proof goes by establishing a control over the length of the orbit $\bz(t)$ near the accumulation point $\bz_0$.  We proceed with an estimate on the arc-length functional, which is a  local version of Simon's result \cite{Simon}. 

Let us denote
\begin{equation*}\label{e:H}
H(t):=\Phi(\bz(t))+\frac{3}{4}\int_{t}^{\infty}\|\bE(s)\|^2\ds.
\end{equation*}

\begin{lemma}\label{l:al}
As long as $\bz(t)\in B_{\delta}(\bz_0)$ for $t'\leq t \leq t''$, we have
\begin{equation*}\label{e:L}
\int_{t'}^{t''}\|\dot{\bz}(s)\|\ds \leq 4\int_{H(t'')}^{H(t')}\frac{1}{c|\xi-\Phi(\bz_0)|^{\mu}}\dxi+\int_{t'}^{t''}\tilde{E}(s)\ds
\end{equation*}
where $\tilde{E}$ is an exponentially decaying quantity.
\end{lemma}
\begin{proof}
We have
\begin{equation}\label{dotH}
-\dot{H}(t)=-\langle\n\Phi(\bz(t)),\dot{\bz}(t)\rangle+\frac{3}{4}\|\bE(t)\|^2\geq  \frac{1}{4}\|\n\Phi(\bz(t))\| \|\dot{\bz}(t)\|.
\end{equation}
So,  the  function $H(\cdot)$  is  non-increasing. To continue, we define another auxiliary function
\[
\Psi(x):=\int_{0}^{x}\frac{1}{\psi(\xi)}\dxi,
\]
where
\[
\psi(\xi):=c|\xi-\Phi(\bz_0)|^{\mu}.
\]
Since $\mu <1$, we have 
\begin{equation}\label{e:psimu}
\psi(H(t))\leq c|\Phi(\bz(t))-\Phi(\bz_0)|^{\mu}+c\left( \frac34 \int_{t}^{\infty}\|\bE(s)\|^2\ds\right)^{\mu}.
\end{equation}

Let us compute 
\[
\ddt \Psi(H(t))=\dot{\Psi}(H(t))\dot{H}(t)=\frac{\dot{H}(t)}{\psi(H(t))}.
\]
Combining with \eqref{dotH} and \eqref{e:psimu}, 
\[
-\ddt \Psi(H(t))\geq  \frac{1}{4c }\cdot \frac{  \|\n\Phi(\bz(t))\| \|\dot{\bz}(t)\|}{|\Phi(\bz(t))-\Phi(\bz_0)|^{\mu}+ \left( \frac{3}{4}\int_{t}^{\infty}\|\bE(s)\|^2\ds\right)^{\mu}},
\]
and thus,  by the Lojasiewicz gradient inequality \eqref{Lojasiewicz} we get
\[
-\ddt \Psi(H(t))\geq \frac{1}{4} \cdot \frac{\|\n\Phi(\bz(t))\| \|\dot{\bz}(t)\|}{\|\n\Phi(\bz(t))\|+c\left(\frac34 \int_{t}^{\infty}\|\bE(s)\|^2\ds\right)^{\mu}} \qquad \text{for all } t\in (t',t'').
\]
Hence, for all $t\in (t',t'')$ we obtain
\begin{equation}\label{step1}
-4 \ddt \Psi(H(t))\geq \|\dot{\bz}(t)\|-c\left( \frac34 \int_{t}^{\infty}\|\bE(s)\|^2\ds\right)^{\mu}\frac{ \|\dot{\bz}(t)\|}{\|\n\Phi(\bz(t))\|+c\left( \frac34 \int_{t}^{\infty}\|\bE(s)\|^2\ds\right)^{\mu}}.
\end{equation}
To estimate the second term on the right-hand side we use the system \eqref{e:fullZ} to get
\begin{equation}\label{step2}
\frac{ \|\dot{\bz}(t)\|}{\|\n\Phi(\bz(t))\|+c\left( \frac34 \int_{t}^{\infty}\|\bE(s)\|^2\ds\right)^{\mu}}\leq 1+\frac{\|\bE(t)\|}{c\left( \frac34 \int_{t}^{\infty}\|\bE(s)\|^2\ds\right)^{\mu}}.
\end{equation}
Combining \eqref{step1} and \eqref{step2} gives 
\[
-4 \ddt \Psi(H(t))\geq \|\dot{\bz}(t)\|-c \left( \frac34 \int_{t}^{\infty}\|\bE(s)\|^2\ds\right)^{\mu}-\|\bE(t)\| \qquad \text{for all } t\in (t',t'').
\]
Denoting
\begin{equation*}\label{e:tildeE}
\tilde{E}(t):=c\left( \frac34\int_{t}^{\infty}\|\bE(s)\|^2\ds\right)^{\mu}+\|\bE(t)\|,
\end{equation*}
we obtain
\begin{equation}\label{step3}
\|\dot{\bz}(t)\|\leq -4 \ddt \Psi(H(t))+\tilde{E}(t) \qquad \text{for all } t\in (t',t'').
\end{equation}
Integrating \eqref{step3} over $(t',t'')$ finishes the proof.\end{proof}

\begin{proof}[Conclusion of the proof of \thm{t:unique}]
To conclude the proof of convergence we argue as follows. Let us fix an arbitrary $r<\d$ and consider a remote time $t_n\gg 1$ such that 
\[
\bz(t_n)\in B_{\frac{r}{3}}(\bz_0), \qquad 4\int_{\Phi(\bz_0)}^{H(t_n)}\frac{1}{c|\xi-\Phi(\bz_0)|^{\mu}}\dxi <\frac{r}{3}, \qquad \int_{t_n}^{\infty}\tilde{E}(s)\ds < \frac{r}{3}.
\]
We show that the entire trajectory for $t> t_n$ lies in $B_r(\bz_0)$. By contradiction, suppose not,
and let $t_n+\tilde{t}$ be the smallest $\tilde{t} > 0$ such that $\|\bz(t_n+\tilde{t})-\bz_0\|=r$. Then $\bz(t)$ lies in $B_r(\bz_0)$ for all $t\in(t_n,t_n+\tilde{t}).$ So, applying \lem{l:al} with $t'=t_n$ and $t''=t_n+\tilde{t}$ we obtain
\[
\|\bz(t_n+\tilde{t})-\bz_0\| \leq \|\bz(t_n+\tilde{t})-\bz(t_n)\|+\|\bz(t_n)-\bz_0\|\leq \int_{t_n}^{t_n+\tilde{t}}\|\dot{\bz}(s)\|\ds +\frac{r}{3}<r
\]
which is a contradiction. Thus $\bz(t)$ remains in $B_r(\bz_0)$ for all $t\in[t_n,\infty)$.

To conclude that $\bz_0$ is the equilibrium we note that the above argument implies that  
\[
\int_{t_n}^{\infty}\|\dot{\bz}(s)\|\ds <\infty.
\]
Thus, $\dot{\bz}(s_n) \to 0$, and hence $\n \Phi(\bz(s_n)) \to 0 = \n \Phi(\bz_0)$.

The proof of \thm{t:unique} is complete. 
\end{proof}

\subsection{Further structural properties of Nash equilibria}

Although we can't compute the equilibrium $\by^*$ explicitly, certain structural properties of it can be provided. First, there is an order of shifts of $y_i$'s relative to $\th_i$'s. 

\begin{lemma}
	There exists an index $1\leq i_0\leq N$ such that 
	\begin{equation*}
		\begin{split}
		\th_1 \leq y_1^p, &\ \ldots\ , \th_{i_0} \leq y_{i_0}^p,\\
		y_{i_0} \leq&\ \bar{y} \leq y_{i_0 + 1},\\
		y_{i_0 + 1}^p \leq \th_{i_0 + 1}, &\ \ldots\ , y_N^p \leq \th_N.
		\end{split}
	\end{equation*}
\end{lemma}
In other words, the opinions below average undergo right shift from agent's convictions, and opinions above average 
undergo left shift.
\begin{proof}
	The result follows by rewriting \eqref{e:yi} as follows
\[
	\frac{M \bar{y} + \s \th_i y_i}{M+ \s \th_i} = \frac{M+\s y_i^p  }{M+\s \th_i}y_i.
\]
So, if $y_i^p \leq \th_i$, then
\[
\frac{M \bar{y} + \s \th_i y_i}{M + \s \th_i} \leq y_i,
\]
which implies $y_i \geq \bar{y}$. Similarly if $y_i^p \geq \th_i$, then $y_i \leq \bar{y}$. So, both implications can be reversed as well. By monotonicity of $y_i$'s, there exists a unique index  $i_0$  for which 
\[
y_1\leq \ldots \leq y_{i_0} \leq \bar{y}  \leq y_{i_0 + 1} \leq \ldots \leq y_N.
\]
 Then the shift inequalities hold as stated.  
\end{proof}

Next, asymptotic behavior of Nash equilibria  can be described in the limit as the friction coefficient $\s \to \infty$ or $\s \to 0$. Indeed, if $\s \to \infty$, directly from   \eqref{e:yi} we can see that any solution converges to its conviction values
$ \by^*(\s)  \to \bth^{1/p}$.

On the other hand, as $\s \to 0$, the alignment force becomes dominant as it is expected that the opinions will reach a consensus. That consensus can be determined as follows.  Since the coordinates of $\by^*$ remain within the fixed bounds of $\th_1, \th_N$ as we let $\s \to 0$, we can see directly from \eqref{e:yi} that for all $i =1,\ldots,N$
\[
 \left| y_i(\s) - \bar{y}(\s) \right|  \to 0.
\]
It remains to determine the limit of $\bar{y}(\s)$. So, adding up all the equations in system \eqref{e:yi} premultiplied by the masses $m_i$ we obtain, in the notation of \eqref{e:m-prod},
\[
\lan \by^*(\s), \bth \ran = \| \by^*(\s)\|_{p+1}^{p+1}.
\]
Thus, as $\s \to 0$, we have
\[
\bar{y}(\s) \bar{\th}  = \bar{y}(\s)^{p+1} + o(1).
\]
Since the average $\bar{y}(\s)$ stays uniformly bounded from zero, we obtain 
\[
\bar{\th}  = \bar{y}(\s)^{p} + o(1),
\]
and hence,
\[
\bar{y}(\s)\to  \bar{\th}^{1/p}.
\]
We summarize the above in the following lemma.

\begin{lemma}
We have the following asymptotic behavior of the equilibrium values:
\begin{align*}
\lim_{\s \to \infty} y_i(\s)  &= \th_i^{1/p}, \qquad \forall i = 1,\ldots,N,\\
\lim_{\s \to 0} y_i(\s)  &=  \bar{\th}^{1/p}, \qquad \forall i = 1,\ldots,N.
\end{align*}
\end{lemma}

%\bibliographystyle{plain}
%\bibliography{collective-pure,collective-appl,friction}

\end{document}